\documentclass{article}
\usepackage[utf8]{inputenc}
\usepackage{graphicx}
\usepackage{float}
\usepackage[english]{babel}
\usepackage{amsmath}
\usepackage{amsthm}
\usepackage{amsfonts}
\usepackage{hyperref}
\usepackage{geometry}
\usepackage[colorinlistoftodos]{todonotes}
\usepackage[capitalize]{cleveref}
\usepackage{tikz}
\usetikzlibrary{calc}

\newtheorem{theorem}{Theorem}[section]
\newtheorem{corollary}{Corollary}[theorem]
\newtheorem{lemma}[theorem]{Lemma}
\newtheorem{observation}[theorem]{Observation}

\newtheorem{proposition}{Proposition}[section]
\newtheorem{conjecture}{Conjecture}[section]
\newtheorem{Problem}{Problem}[section]
\newtheorem{question}{Question}[section]
\theoremstyle{definition}
\newtheorem{definition}{Definition}[section] %defines definition

\DeclareMathOperator{\Av}{Av}
\DeclareMathOperator{\VHC}{VHC}
\DeclareMathOperator{\RedVHC}{RedVHC}
\DeclareMathOperator{\Dyck}{Dyck}
\DeclareMathOperator{\Duck}{Duck}
\DeclareMathOperator{\vertc}{vert}
\DeclareMathOperator{\hor}{hor}
\DeclareMathOperator{\Hooks}{Hooks}

\title{312-Avoiding Reduced Valid Hook Configurations and Duck Words}
\author{Ilani Axelrod-Freed}
\date{October 2020}

\begin{document}

\maketitle
\begin{abstract}
    Valid hook configurations are combinatorial objects used to understand West's stack sorting map as well as cumulants in noncommutative probability theory. We show a bijection between reduced valid hook configurations on 312-avoiding permutations with the maximal allowed number of points and 3D-Dyck words, proving a conjecture of Sankar's. We extend to a bijection between all 312-avoiding reduced valid hook configurations and 3D-Dyck words with specified modifications. We show how these can be counted in terms of the number of 3D-Dyck words of length $3k$ in which exactly $i$ Y's do not have an X immediately before them, the $(k,i)$-Duck words, and use this relationship to prove several properties about sums of 312-avoiding reduced valid hook configurations, including two more of Sankar's conjectures. We also show that the number of $(k,1)$-Duck words is given by a variant of the  tennis ball numbers.
\end{abstract}

\section{Introduction}
\subsection{Motivation}

Defant \cite{postorder, preimages} introduced valid hook configurations to count the cardinality of the preimage of permutations under the stack sorting map, called its fertility. Defant has used his Fertility Formula to prove many results about the stack sorting map. 

Following this, Defant, Engen, and Miller established an unexpected bijection between valid hook configurations and certain weighted set partitions that appear in free probability theory (also called noncommutative probability theory) \cite{stacksorting}. In \cite{troupes}, Defant formalized the connection between valid hook configurations and noncommutative probability, giving a simple combinatorial formula that converts from free cumulants to classical cumulants. The formula is given by a sum over valid hook configurations. Valid hook configurations serve as a bridge that allows one to use tools from free probability theory in order to prove deep facts about the stack sorting map, as well as the other way around. Defant also used valid hook configurations and a generalization of stack sorting to prove a surprising combinatorial result about free cumulants and classical cumulants involving colored binary plane trees \cite{troupes}. They can also be used to generalize many of the results known about stack sorting to a much more general context involving special families of colored binary plane trees called troupes. 

In \cite{troupes} Defant gives a formula to convert free cumulants to classical cumulants which is given by summing over 231-avoiding valid hook configurations. Defant's Fertility Formula also tells us that 231-avoiding valid hook configurations are closely related to 2-stack-sortable permutations, which have been studied extensively. This motivated the study of 231-avoiding valid hook configurations and from there, further pattern-avoiding valid hook configurations. In particular, 312-avoiding valid hook configurations are in bijection with intervals in partial orders on Motzkin paths. Sankar \cite{maya} proved that the 312-avoiding valid hook configurations are also in bijection with certain closed lattice walks and used this to show that their generating function is not D-finite. She shows that there is also a formula for the number of 312-avoiding valid hook configurations in terms of 312-avoiding reduced valid hook configurations, counted according to the number of hooks and the number of points. We will prove properties about the number of 312-avoiding reduced valid hook configurations, settling several conjectures of Sankar's.

\subsection{Sankar's Conjectures}

%%%%Sankar's sum expressing 312-avoiding valid hook configurations in terms of reduced 312-avoiding reduced valid hook configurations with a given number of points in the permutation and a given number of hooks drawn on motivates further study of reduced 312-avoiding valid hook configurations with $k$ hooks and $n$ points.  

 Let $\RedVHC_k(\Av_{3k-i}(312))$ denote the set of $312$-avoiding reduced valid hook configurations with $k$ hooks that are on permutations with $2k+i$ points. We will define these notions in Section 2. For now,
 we simply want to consider the following triangle of numbers, where the entry in the $k^{th}$ row and $i^{th}$ column is $|\RedVHC_k(\Av_{2k+i}(312))|$:

$\begin{array}{ccccccc}
1 & \ & \ & \ & \  & \ & \ \\
3 &   5 & \ & \ & \  & \ & \ \\
14  &  51  &   42  & \  & \  & \ & \ \\
84  &  485 &     849   &    462  & \  & \ & \ \\
594  & 4743   &  13004 &  14819   &   6006  & \ & \ \\
4719  & 48309   &   182311   &  322789   &   271452 &   87516 & \ \\
40898 & 511607 & 2472322 & 5999489 & 7794646 & 5182011 &
 138567 .
\end{array}$

\medskip

\begin{definition}
Let
$$f_k(x)=\sum_{i=0}^{k-1}|\RedVHC_k(\Av_{3k-i}(312))| \cdot x^i$$
$$\text{and } h_k(x) = f_k(x-1).$$
\end{definition}

There are many patterns observable in the triangle of numbers related to the coefficients of these polynomials. Sankar has conjectures in \cite{maya} about these numbers:

\begin{conjecture}
For all $k \geq 1$, $f_k(0) = h_k(1) = |\RedVHC_k(\Av_{3k}(312))| = 2\frac{(3k)!}{k!(k+1)!(k+2)!}$, the $k^{th}$ three dimensional Catalan number.
\label{conj: Sankar 1}
\end{conjecture}

\begin{conjecture}
For every $k \geq 1$, $f_k(-1) = h_k(0) = C_k$.
\label{conj: Sankar 2}
\end{conjecture}

\begin{conjecture}
The polynomial $h_k(x)$ has strictly positive coefficients.
\label{conj: Sankar 3}
\end{conjecture}
    
\begin{conjecture}
The linear coefficient of $h_k(x)$ is given by the $(k-1)^{st}$ weighted tennis ball number.
\label{conj: Sankar 4}
\end{conjecture}

In Section 3 we establish some basic properties of 312-avoiding valid hook configurations. In Section 4 we define 3D-Dyck words and prove \cref{conj: Sankar 1} by showing these words are in bijection with 312-avoiding reduced valid hook configurations with the maximal possible number of points. In Section 5 we generalize this bijection to all 312-avoiding reduced valid hook configurations and 3D-Dyck words in which a certain number of Y's that do not have X's immediately before them are underlined, called underlined duck words. In Section 6 we give a formula for the number of these words in terms of the number of duck words -- 3D-Dyck words in which exactly $i$ Y's do not have an X immediately before them. This provides a way to express all 312-avoiding valid hook configuration and thus intervals on certain posets defined on Motzkin paths in terms of the more intuitively-defined duck words. We use this to prove Conjectures \ref{conj: Sankar 2} and \ref{conj: Sankar 3} as well as further properties. In particular we show that the coefficients of $h_k(x)$ are given by the duck words, a far more elegant combinatorial interpretation for $f_k(x)$ than what we started with. In Section 7 we count the number of 3D-Dyck words where exactly one Y does not have an X immediately before it and show that this proves \cref{conj: Sankar 4}.

\section{Preliminaries}
\subsection{Definitions}
For the following definitions, let $S_n$ be the set of permutations of length $n$, and $\pi = \pi_1 \cdots \pi_n \in S_n$ be a permutation.

\begin{definition}
The permutation $\pi$ is called ${\sigma}$\textit{-avoiding} if it does not contain a subsequence that is order isomorphic to the permutation $\sigma$. In other words, for $\sigma \in S_k$, $\pi$ is $\sigma$-avoiding if there is no length $k$ subsequence of $\pi$, written with increasing indices as $\pi_{a_1} \cdots \pi_{a_k}$, such that $\pi_{a_i}>\pi_{a_j}$ exactly when $\sigma_i>\sigma_j$.
\end{definition}

We let $\Av_n(\sigma)$ be the set of permutations in $S_n$ that avoid $\sigma$.

\begin{definition}
The \textit{plot} of the permutation $\pi$ is the set of points $\{(i, \pi_i): i \in [n] \} \subseteq \mathbb{R}^2$. 
\end{definition}

\begin{definition}
A descent occurs when $\pi_i>\pi_{i+1}$. Then $i$ is the descent and we call $(i,\pi_i)$ the \textit{descent top} and $(i+1, \pi_{i+1})$ the \textit{descent bottom}.
\end{definition}

We construct a hook on $\pi$ by drawing a vertical line up from a point $(i, \pi_i)$, and then a horizontal line to another point $(j, \pi_j)$ where $i<j$ and $\pi_i < \pi_j$ as illustrated in \cref{figure: a hook}.

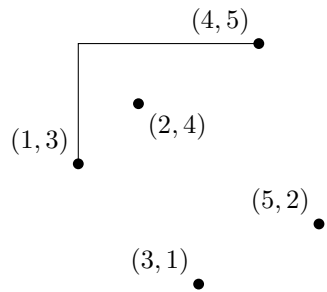
\begin{figure}[h]
\begin{center}
\begin{tikzpicture}[scale=.8]
\fill (1, 3) circle (.9mm) node[above left] {$(1,3)$} (2, 4) circle (.9mm) node[below right] {$(2,4)$} (3, 1) circle (.9mm) node[above left] {$(3,1)$} (4, 5) circle (.9mm) node[above left] {$(4,5)$} (5,2) circle (.9mm) node[above left] {$(5,2)$};
\draw (1,3) --(1,5) -- (4,5);
\end{tikzpicture}
\caption{The plot of $\pi=34152$ with a hook from $(1,3)$ to $(4,5)$.}
\label{figure: a hook}
\end{center}  
\end{figure}

\begin{definition}
For a hook with points as above, the point $(i, \pi_i)$ is the \textit{southwest (SW) endpoint} and $(j,\pi_j)$ is the \textit{northeast (NE) endpoint}.
\end{definition}

\begin{definition}
A \textit{valid hook configuration} on a permutation $\pi$ is a set of hooks such that
\begin{enumerate}
  \item[(i)] The set of southwest hook endpoints is exactly the set of descent tops.

\item[(ii)] No point in the plot of $\pi$ lies above a hook.

\item[(iii)] Hooks do not intersect, except at endpoints. \\
See Figures \ref{fig:valid HC} and \ref{fig:invalid HC} for examples of valid and invalid hook configurations respectively.
\end{enumerate}

For $S \subset S_n$, we let $\VHC(S)$ be the set of valid hook configurations on permutations in $S$, and $\VHC_k(S)$ be the subset of those valid hook configurations that have exactly $k$ hooks.
\label{Def:VHC}
\end{definition}

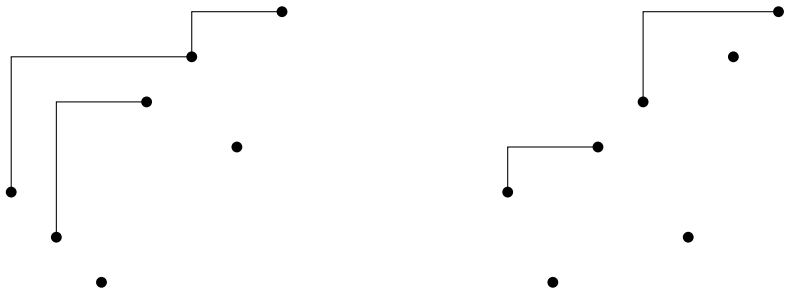
\begin{figure}[H]\begin{center}
\begin{tikzpicture}[scale=.6]
	\coordinate (P1) at (1,3);
	\coordinate (P2) at (2,2);
	\coordinate (P3) at (3,1);
	\coordinate (P4) at (4,5);
	\coordinate (P5) at (5,6);
	\coordinate (P6) at (6,4);
	\coordinate (P7) at (7,7);
	\foreach \pt in {P1, P2, P3, P4, P5, P6, P7}
		\fill (\pt) circle (1.2mm);
	\foreach \a/\b in {P1/P5, P2/P4, P5/P7}
		\draw (\a) -- (\a)|-(\b) -- (\b);
	\coordinate (Q1) at (12,3);
	\coordinate (Q2) at (13,1);
	\coordinate (Q3) at (14,4);
	\coordinate (Q4) at (15,5);
	\coordinate (Q5) at (16,2);
	\coordinate (Q6) at (17,6);
	\coordinate (Q7) at (18,7);
	\foreach \pt in {Q1, Q2, Q3, Q4, Q5, Q6, Q7}
		\fill (\pt) circle (1.2mm);
	\foreach \a/\b in {Q1/Q3, Q4/Q7}
		\draw (\a) -- (\a)|-(\b) -- (\b);
\end{tikzpicture}
\caption{Valid hook configurations on the permutations $3215647$ and $3145267$.} 
\label{fig:valid HC}
\end{center}\end{figure}
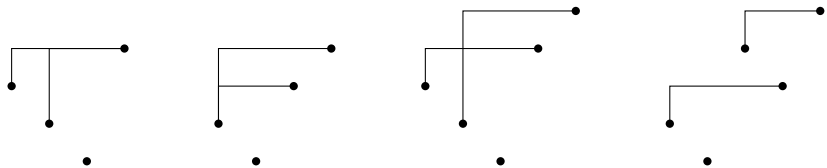
\begin{figure}[h]\begin{center}
\begin{tikzpicture}[scale=.5]
	\coordinate (S2) at (0,0);
	\coordinate (S3) at (5.5,0);
	\coordinate (S1) at ($(S3) + (5.5,0)$);
	\coordinate (S4) at ($(S1) + (6.5,0)$);
	\coordinate (P1) at (1,3);
	\coordinate (P2) at (2,2);
	\coordinate (P3) at (3,1);
	\coordinate (P4) at (4,4);
	\coordinate (P5) at (5,5);
	\coordinate (Q1) at (1,2);
	\coordinate (Q2) at (2,1);
	\coordinate (Q3) at (3,3);
	\coordinate (Q4) at (4,4);
	\coordinate (R3) at (3,4);
	\coordinate (R4) at (4,3);
	\coordinate (R5) at (5,5);
	\foreach \pt in {P1, P2, P3, P4, P5}
		\fill ($ (\pt) + (S1) $) circle (1.1mm);
	\foreach \pt in {P1, P2, P3, P4}
		\fill ($ (\pt) + (S2) $) circle (1.1mm);
	\foreach \pt in {Q1, Q2, Q3, Q4}
		\fill ($ (\pt) + (S3) $) circle (1.1mm);
	\foreach \pt in {Q1, Q2, R3, R4, R5}
		\fill ($ (\pt) + (S4) $) circle (1.1mm);
	\foreach \a/\b/\s in {P1/P4/S1, P2/P5/S1, P1/P4/S2, P2/P4/S2, Q1/Q3/S3, Q1/Q4/S3, Q1/R4/S4, R3/R5/S4}
		\draw ($(\a)+(\s)$) -- ($(\a)+(\s)$)|-($(\b)+(\s)$) -- ($(\b)+(\s)$);
\end{tikzpicture}
\caption{Hook configurations on permutations failing conditions (ii) or (iii) of \cref{Def:VHC}.}
\label{fig:invalid HC}
\end{center}\end{figure}

\begin{definition}
A valid hook configuration is said to be \textit{reduced} if every point is a hook endpoint or descent bottom (or both).

For $S \subseteq S_n$, we let $\RedVHC(S)$ be the set of reduced valid hook configurations on permutations in $S$, and $\RedVHC_k(S)$ be the the subset of reduced valid hook configurations that have exactly $k$ hooks.
\end{definition}

Note that the valid hook configuration on the left in \cref{fig:valid HC} is reduced while the one on the right is not. Indeed in the valid hook configuration on the right, the point $(6,6)$ is neither an endpoint nor a descent bottom.

\bigskip

%For a permutation $\pi=\pi_1 \cdots \pi_n$ we can remove some number of points and then \textit{renormalize} it to get a new permutation $\sigma=\sigma_1 \cdots \sigma_m$, where $m$ is the number of remaining points, $\sigma_i \in [m]$ for all $i$, and $\sigma_1 \cdots \sigma_m$ is order isomorphic to the subsequence of $\pi$ that consists exactly of all points that were not removed.

We now define Dyck words and 3D-Dyck words, which will arise in several bijections later on:

\begin{definition} 
A \textit{Dyck word} of length $2k$ is a word with $k$ U's and $k$ D's, where up to any point in the word the number of U's is greater than or equal to the number of D's. We let $\Dyck_k$ denote the set of Dyck words of length $2k$. 
\end{definition}

The number of Dyck words of length $2k$ is given by $C_k = \frac{1}{k+1}\binom{2k}{k}$, the $k^{th}$ Catalan number.

\begin{definition}
A \textit{3D-Dyck word} of length $3n$ is a word that has $n$ X letters, $n$ Y letters, and $n$ Z letters, where at any point in the word the number of X's is always greater than or equal to the number of Y's, which in turn is greater than or equal to the number of Z's. Let $\Dyck^3_k$ be the set of 3D-Dyck words of length $3k$. 
\end{definition}

The number of 3D-Dyck words of length $3k$ is given by the $k^{th}$ 3D-Catalan number, $2\frac{(3k)!}{k!(k+1)!(k+2)!}$.

\subsection{Basic Properties}
We first establish some basic rules and properties that apply to all 312-avoiding valid hook configurations, then use these to prove further properties about 312-avoiding reduced valid hook configurations.
\medskip

The \emph{normalization} of a sequence of $n$ positive integers is the permutation obtained by replacing the $i^{\text{th}}$-smallest entry with $i$ for all $i$. For example, the normalization of $3759$ is $1324$. Given a permutation or a valid hook configuration, we can remove some number of points and then similarly normalize.

Specifically we can remove any point that is neither a descent bottom nor an endpoint from a valid hook configuration to obtain something that can be normalized to get a new valid hook configuration. If all such points are removed, we obtain a reduced valid hook configuration. Doing this to a 312-avoiding valid hook configuration preserves 312-avoidance. 

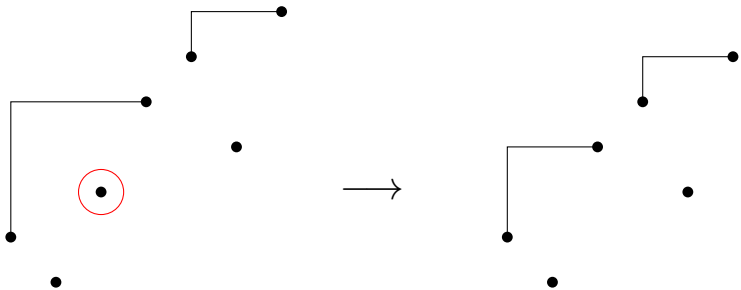
\begin{figure}[H]\begin{center}
\begin{tikzpicture}[scale=.6]
	\coordinate (Q1) at (1,2);
	\coordinate (Q2) at (2,1);
	\coordinate (Q3) at (3,3);
	\coordinate (Q4) at (4,5);
	\coordinate (Q5) at (5,6);
	\coordinate (Q6) at (6,4);
	\coordinate (Q7) at (7,7);
	\foreach \pt in {Q1, Q2, Q3, Q4, Q5, Q6, Q7}
		\fill (\pt) circle (1.2mm);
	\foreach \a/\b in {Q1/Q4, Q5/Q7}
		\draw (\a) -- (\a)|-(\b) -- (\b);
	\draw[red] (Q3) circle (5mm);
	\draw (9,3) node[scale=1.5] {$\longrightarrow$};
	\coordinate (P1) at (12,2);
	\coordinate (P2) at (13,1);
	\coordinate (P3) at (14,4);
	\coordinate (P4) at (15,5);
	\coordinate (P5) at (16,3);
	\coordinate (P6) at (17,6);
	\foreach \pt in {P1, P2, P3, P4, P5, P6}
		\fill (\pt) circle (1.2mm);
	\foreach \a/\b in {P1/P3, P4/P6}
		\draw (\a) -- (\a)|-(\b) -- (\b);
\end{tikzpicture}
\caption{Left: A 312-avoiding valid hook configuration on $2135647$. Right: The resulting 312-avoiding reduced valid hook configuration after the point circled in red is removed and the configuration is normalized.} 
\label{fig:valid HC}
\end{center}\end{figure}

Using this fact, Sankar showed in \cite{maya} that the relationship between reduced and non-reduced 312-avoiding valid hook configurations is given by a bijection

\begin{equation}\label{Eq1}\VHC(\Av_n(312))\to\bigcup_{r=0}^n\RedVHC(\Av_r(312))\times {[n]\choose r}.
\end{equation}

In a reduced valid hook configuration with $k$ hooks, the possible number of points is $3k-i$, where $0 \leq i < k$. The maximum number of points is obtained when each point is exactly one of descent bottom, SW endpoint, and NE endpoint, making one of each type of point for each hook. The minimum is obtained when every SW endpoint after the first is also simultaneously either a descent bottom or NE endpoint. This range for $i$ justifies why $f_k$ and $h_k$ are defined as they are.

\begin{observation} In order for a permutation to be 312-avoiding, any time there is a descent top, $(i, \pi_i)$, followed by a descent bottom, $(i+1, \pi_{i+1})$, we must have that $\pi_{i+1}$ is the largest number less than $\pi_i$ not contained in $\{\pi_1,...,\pi_{i-1}\}$. 
\end{observation}

We can use this to determine the heights of each individual descent bottom of a permutation given sufficient information about the permutation. For the following, we will let $t_i$ denote the height of the $i^{th}$ descent top and $b_i$ denote the height of the $i^{th}$ descent bottom within a given permutation.

\begin{definition}
For a permutation $\pi \in Av(312)$,  \\
$$BH_{i,\pi}=\{b_j: j \geq i\}.$$
And let $BH_\pi=BH_{1,\pi}$ denote the full set of heights of descent bottoms of $\pi$. If $C$ is a valid hook configuration on a permutation $\pi$, we let $BH_{i,C}=BH_{i,\pi}$
\end{definition}

\begin{lemma}
\label{lemma:BH_C heights}
For any permutation $\pi$, we have that $b_i$ is determined by $t_i$ and the set $BH_{i,\pi}$. Then $b_i$ is uniquely determined for all possible $i$ by the set of descent top heights and the set $BH_\pi$.
\end{lemma}

\begin{proof}
We have that $b_i<t_i$ and $b_i \in BH_{i,\pi}$. Suppose we chose $b_i$ such that there was some $b \in BH_{i,\pi}$ with $b_i < b < t_i$. Then $\pi$ would contain $t_i$ followed by $b_i$ followed at some point by $b$, which would be a 312. Therefore $b_i$ must be the largest element in $BH_{i,\pi}$ less than $t_i$, so its height is uniquely determined.

We can now prove uniqueness using induction. \\
\textit{Base Case:} There is exactly one choice for $b_1$, namely the largest element in $BH_\pi$ less than $t_1$. \\
\textit{Inductive Step:} Assume that the heights of $b_1, ..., b_i$ are uniquely determined. Then $BH_{i+1, \pi} = BH_\pi \setminus \{b_1,...,b_i\}$. Knowing $t_{i+1}$ and $BH_{i+1, \pi}$ now uniquely determines $b_{i+1}$. 
\end{proof}

%%%%%%%%%%%%%%%%%%%%%%%%%%%%%%%%%%%%%%%%%%%%%%%%%%%%%%%%%
\section{$|\RedVHC_k(\Av_{3k}(312))|$}
Having established basic properties of $\VHC(\Av(312))$, we now examine 312-avoiding reduced valid hook configurations with $k$ hooks and the maximal number of points. We will prove further properties about these configurations, leading to a bijection with 3D-Dyck words. This will prove \cref{conj: Sankar 1}. We begin by describing the hook endpoints.

A \emph{left-to-right maximum} of a permutation $\pi=\pi_1\cdots\pi_n$ is a point $(i,\pi_i)$ such that $\pi_j<\pi_i$ for all $1\leq j<i$. 

\begin{lemma}
\label{lemma:rising endpts} Let $C\in\RedVHC_k(\Av_{3k}(312))$ be a $312$-avoiding reduced valid hook configuration, and let $\pi\in S_{3k}$ be its underlying permutation. Every hook endpoint of $C$ is a left-to-right maximum of $\pi$. 
\end{lemma}

\begin{proof} Because $C$ has $3k$ points and $k$ hooks, each of its points must be exactly one of the following:

\begin{enumerate}
    \item Descent bottom 
    
    \item SW endpoint/descent top
    
    \item NE endpoint.
\end{enumerate}
Suppose by way of contradiction that there is some hook endpoint $(i, \pi_i)$ of $C$ that is not a left-to-right maximum of $\pi$. This means that there is a point $(j,\pi_j)$ with $1\leq j<i$ and $\pi_j>\pi_i$. Since $(i,\pi_i)$ is a hook endpoint, it cannot be a descent bottom. Thus, $\pi_{i-1}<\pi_i$. However, this implies that the entries $\pi_j, \pi_{i-1}, \pi_i$ form an occurrence of the pattern $312$ in $\pi$, which is a contradiction.
\end{proof}

\begin{figure}[H]
\begin{center}
\begin{tikzpicture}[scale=.8]
\fill (7.5, 3) circle (.9mm) node[below] {$(i,\pi_i)$} (8.5, 1) circle (.9mm) node[below] {$(i+1,\pi_{i+1})$} (9.5, 2) circle (.9mm) node[below] {$(i+2,\pi_{i+2})$};
\draw (7.5,3) --(7.5,5) -- (11.5,5);
\draw (9.5,2) -- (9.5,4) -- (11,4);
\draw[red] (7.5,3) node[right, scale=1.5] {\textbf{3}};
\draw[red] (8.5,1) node[right, scale=1.5] {\textbf{1}};
\draw[red] (9.5,2) node[right, scale=1.5] {\textbf{2}};
\fill (2, 3) circle (.9mm) node[below] {$(i,\pi_i)$} (3, 2) circle (.9mm) node[below] {$(i+1,\pi_{i+1})$};
\draw (1,1) -- (1,3) -- (2,3);
\draw (2,3) -- (2,4) -- (3.5,4);
\draw[red] (2,3) node[right, scale=2] {\textbf{x}};
\end{tikzpicture}
\caption{Non-allowed hook placement examples discussed in \cref{lemma:rising endpts}. Points that are more than one of descent bottom, SW and NE endpoint are marked with a red X.}
\label{fig: invalid placement}
\end{center}
\end{figure}
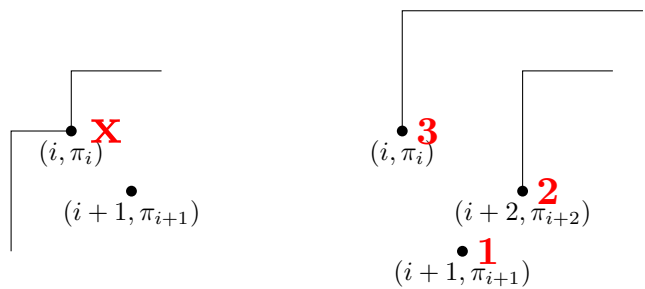
%%%%%%%%%%%%%%%%%%%%%%%%%%%%%%%%%%%%%%
\begin{figure}[H]
\begin{center}
\begin{tikzpicture}[scale=.8]
\fill (1, 2) circle (.9mm) node[below] {$(i,\pi_i)$} (2, 1) circle (.9mm) node[below] {$(i+1,\pi_{i+1})$} (3, 3) circle (.9mm) node[below] {$(i+2,\pi_{i+2})$};
\draw (1,2) --(1,3) -- (3,3);
\fill (7, 2) circle (.9mm) node[below] {$(i,\pi_i)$} (8, 1) circle (.9mm) node[below] {$(i+1,\pi_{i+1})$} (9, 3) circle (.9mm) node[below] {$(i+2,\pi_{i+2})$};
\draw (7,2) --(7,4.5) -- (11,4.5);
\draw (9,3) -- (9,4) -- (10.5,4);
\fill (14.5, 2) circle (.9mm) node[below] {$(i,\pi_i)$} (15.5, 3) circle (.9mm) node[below] {$(i+1,\pi_{i+1})$};
\draw (13.5,1) -- (13.5,2) -- (14.5,2);
\draw (15.5,3) -- (15.5,4) -- (16.5,4);
\end{tikzpicture}
\caption{Allowed hook placement examples.}
\label{fig: possible placement}
\end{center}  
\end{figure}
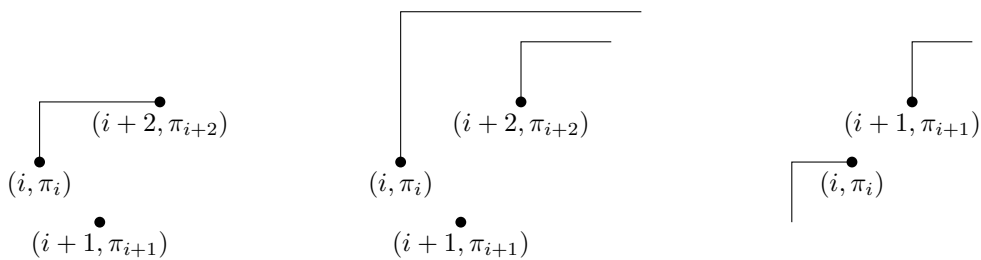

\begin{definition}
For a valid hook configuration $C$, let $C_{\vertc}(i)$ denote the point in $C$ with vertical component $i$, and $C_{\hor}(i)$ denote the point in $C$ with horizontal component $i$. 
\end{definition}

Given a valid hook configuration, $C$, we can remove all points that are not endpoints and then normalize to get a new (non-valid) hook configuration $\Hooks(C)$. For any set of hook configurations $H$, let $\Hooks(H) = \{\Hooks(C) : C \in H\}$.

\begin{corollary}
\label{corollary:dyck bij}
The set $\Hooks(\RedVHC_k(\Av_{3k}(312)))$ is in bijection with Dyck words.
\end{corollary}

\begin{proof}
By \cref{lemma:rising endpts} the $i^{th}$ endpoint in $\Hooks(C)$ is of the form $(i,i)$ for all $i \in [k]$ and $C \in \RedVHC_k(\Av_{3k}(312))$, and is either a SE or NE endpoint (not both).
We can write each SW endpoint as a left parenthesis and each NE endpoint as a right parenthesis going from left to right. Any two parentheses that are matched will be the two endpoints of the same hook. This gives a bijection with correctly-matched parentheses which can be equivalently represented as Dyck words.
\end{proof}

This gives us a full description of the hooks for any configuration in $\RedVHC_k(\Av_{3k}(312))$. Now we just need to show how descent bottoms can be inserted and show that this is consistent with 3D-Dyck words.

\begin{definition}
For $w \in \Dyck^3_k$, let $YZ(w)$ be the word obtained be removing all X's and leaving only the Y's and Z's.
\end{definition}

We observe that YZ($\Dyck^3_k$) = $Dyck_k$.

We are now ready to prove \cref{conj: Sankar 1}.

\begin{theorem} 
\label{thrm:3Dyck bij}
$\RedVHC_k(\Av_{3k}(312))$ is in bijection with $\Dyck^3_k$.
\end{theorem}

\begin{proof}

We define the map $\phi : \RedVHC_k(\Av_{3k}(312)) \rightarrow $ $\Dyck^3$  as follows:
Given a hook configuration $C \in \RedVHC_k(\Av_{3k}(312))$, the $i^{th}$ letter in $\phi(C)$ is
\begin{enumerate}
    \item[X] if $C_{\vertc}(i)$ is a descent bottom
    \item[Y] if $C_{\vertc}(i)$ is a SW endpoint/descent top
    \item[Z] if $C_{\vertc}(i)$ is a NE endpoint.
\end{enumerate}

\cref{fig:phi} illustrates how $\phi$ acts on a given hook configuration.
        
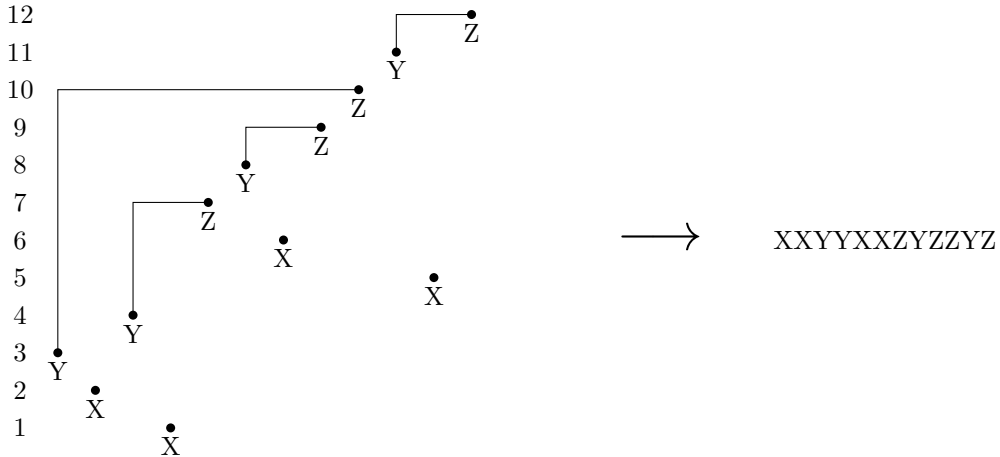
\begin{figure}[H]\begin{center}
\begin{tikzpicture}[scale=.5]
    \draw (0,1) node[scale=1] {1};
    \draw (0,2) node[scale=1] {2};
    \draw (0,3) node[scale=1] {3};
    \draw (0,4) node[scale=1] {4};
    \draw (0,5) node[scale=1] {5};
    \draw (0,6) node[scale=1] {6};
    \draw (0,7) node[scale=1] {7};
    \draw (0,8) node[scale=1] {8};
    \draw (0,9) node[scale=1] {9};
    \draw (0,10) node[scale=1] {10};
    \draw (0,11) node[scale=1] {11};
    \draw (0,12) node[scale=1] {12};
	\coordinate (P1) at (1,3);
	\coordinate (P2) at (2,2);
	\coordinate (P3) at (3,4);
	\coordinate (P4) at (4,1);
	\coordinate (P5) at (5,7);
	\coordinate (P6) at (6,8);
	\coordinate (P7) at (7,6);
	\coordinate (P8) at (8,9);
	\coordinate (P9) at (9,10);
	\coordinate (P10) at (10,11);
	\coordinate (P11) at (11,5);
	\coordinate (P12) at (12,12);
	\foreach \pt in {P1, P2, P3, P4, P5, P6, P7, P8, P9, P10, P11, P12}
		\fill (\pt) circle (1.2mm);
	\foreach \pt in {P2, P4, P7, P11}
		\draw (\pt) node[below, scale=1] {X};
	\foreach \pt in {P1, P3, P6, P10}
		\draw (\pt) node[below, scale=1] {Y};
	\foreach \pt in {P5, P8, P9, P12}
		\draw (\pt) node[below, scale=1] {Z};
	\foreach \a/\b in {P1/P9, P3/P5, P6/P8, P10/P12}
		\draw (\a) -- (\a)|-(\b) -- (\b);
	 \draw (23,6) node {XXYYXXZYZZYZ};
    \draw (17,6) node[scale=2] {$\longrightarrow$};
\end{tikzpicture}
\caption{Conversion from a hook configuration in $\RedVHC_k(\Av_{3k}(312))$ to the 3D-Dyck word XXYYXXZYZZYZ} 
\label{fig:phi}
\end{center}\end{figure}

Ignoring descent bottoms and X's for a moment, \cref{corollary:dyck bij} tells us that \\
$\Hooks(\RedVHC_k(\Av_{3k}(312)))$ is in bijection with $Dyck_k$, so the number of Y's up to any point will never be less than the number of Z's. This still holds when we add back in descent bottoms and X's.

 Below the $i^{th}$ descent top $t_i$ we must have at least $i$ descent bottoms in order for each of the $i$ descent tops (SW endpoints) to have a descent bottom. Recall that the $t_1,...,t_{i-1}<t_i$, so we must also have $b_1,...,b_{i-1}<t_i$. Thus in $\phi(C)$ the number of X's up to any point must be at least the number of Y's, which in turn must be at least the number of Z's. This gives that every hook configuration in $\RedVHC_k(\Av_{3k}(312))$ is sent to a word in $\Dyck^3_k$.

%Suppose that $\phi(C) = \phi(C')$ for some $C,C' \in \RedVHC_k(\Av_{3k}(312))$. Then $YZ(\phi(C))$ must equal $YZ(\phi(C'))$ as we observed above. By \cref{corollary:dyck bij} this happens exactly when $\Hooks(H)=\Hooks(H')$.
% Clearly if $BH_C \neq BH_{C'}$ then the X's will fall differently in $\phi(C)$ and $\phi(C')$, so they will not be equal. Therefore $BH_C = BH_{C'}$. Knowing $\Hooks(C)$ and $BH_C$ gives us the exact heights of all the descent tops for $C$ and $C'$. \cref{lemma:BH_C heights} then tells us that the horizontal components of the descent bottoms must be the same for $C$ and $C'$ as well, so $C=C'$.   

% To finish proving that $\phi$ is a bijection we need only show that 
% $\phi$ is surjective.

To finish the proof, we will show that given $w \in \Dyck^3_k$, we can find the unique $C \in \RedVHC_k(\Av_{3k}(312))$ such that $\phi(C) = w$.

Clearly for any $w \in \Dyck^3_k$, we can use the placement of X's to determine the set $BH_\pi$ for a corresponding valid hook configuration $C$. Any two words with different placement of X's will be sent to different hook configurations.
According to Corollary 3.1.1, $\Hooks(H)$ must also be uniquely determined by YZ$(w)$.
Since the set of descent bottoms is exactly the set of points immediately following each descent top, we can actually determine the exact horizontal components of all hook endpoints.
From there \cref{lemma:BH_C heights} tells us that the height of the $i^{th}$ descent bottom is uniquely determined for all $i \in [k]$. Thus all of $C$ is uniquely determined and is in $\RedVHC_k(\Av_{3k}(312))$.

\end{proof}

This bijection implies the following corollary:

\begin{corollary}
We have $|\RedVHC_k(\Av_{3k}(312))| = 2 \frac{(3k)!}{k!(k+1)!(k+2)!}$.
\end{corollary}

\section{Bijection between $\RedVHC_k(\Av(312))$ and underlined Duck words}
We will now generalize $\phi$ to get a bijection from the set of all 312-avoiding reduced valid hook configurations to the set of underlined duck words.

\begin{definition}
\textit{Underlined duck words} are words in $\Dyck^3$ with the following properties: \\
1. Some number (possibly zero) of the Y's in the word are underlined.\\
2. Underlined Y's do not have an X immediately before them.
\end{definition}

For example XYXXZ\underline{Y}YZXZ\underline{Y}Z and XYXZYZ are underlined duck words while XYZX\underline{Y}Z is not because the underlined Y has an X immediately before it.

We let $\underline{\Duck}_k$ be the number of underlined duck words of length $3k$. A \emph{$(k,i)$-underlined duck word} is a duck word with $3k$ letters and $i$ underlines. We denote the number of such words by $\underline{\Duck}_{k,i}$. This terminology comes from a rewriting rule which we will give in \cref{section:rewritten}.

\begin{theorem}
\label{thrm:duck words bij}
The set $\RedVHC_k(\Av(312))$ is in bijection with underlined duck words of length $3k$. Specifically, for each $i$, $\RedVHC_k(\Av_{3k-i}(312))$ is in bijection with $(k,i)$-underlined duck words.
\end{theorem}

\begin{proof}
Given a hook configuration $C \in \RedVHC_k(\Av_{3k-i}(312))$, we can obtain a new hook configuration $C'$ by doing the following:

\begin{itemize}
\item For each point $p = (i, \pi_i)$ that is both a SW endpoint and either a NW endpoint or descent bottom, insert a new point $p'$ with horizontal component $i+1$ and vertical component one higher than the previous hook endpoint. In the case where $p$ is a NE endpoint of a hook, $p'$ has vertical component $\pi_i+1$. 
\item Let the hook whose southwest endpoint used to be $p$ now have southwest endpoint $p'$.  
\end{itemize}

Every hook configuration in $\RedVHC_k(\Av_{3k-i}(312))$ must have exactly $i$ points that are both a southwest endpoint and something else, so each $C'$ must indeed have $3k$ points and $k$ hooks.  We know $C'$ is a reduced hook configuration because all points that were added were southwest endpoints, and each point that used to be southwest endpoint but is not anymore is still either a descent bottom or northeast endpoint.

Any SW endpoint that was below the previous hook endpoint in $C$ must have been a descent bottom. In $C'$ that hook's SW endpoint is now above that of the previous hook endpoint, so all hook endpoints are left to right maxima. For any inserted SW endpoint $p' = (j, \pi_j)$ we also have a point $(l, \pi_j-1)$ for some $l<j$ (this is the point $p$ that $(j, \pi_j)$ split off from). If $(j, \pi_j)$ starts a 312, then $(l, \pi_j -1)$ would also be in a 312, so $C$ would have had to contain a 312. Since $C$ is 312-avoiding, $C'$ must be as well, so $C' \in \RedVHC_k(\Av_{3k}(312))$.

For a hook configuration $C \in \RedVHC_k(\Av_{3k-i}(312))$, we obtained $C'$ by inserting $i$ points.

%\begin{lemma}
%\label{lemma:reconstruct C}
%Given $C'$ together with the list of points which were inserted we can reconstruct $C$.
%\end{lemma}

%\begin{proof}
%To recover $C$ we just remove all inserted points. For each hook whose southwest endpoint had been an inserted point, $p'$, it's new southwest endpoint is now $p$, the point whose horizontal component is one less than $p'$.
%\end{proof}

Let $\phi'(C)$ be the word $\phi(C')$ with each Y underlined that corresponds to one of the inserted points. \cref{fig:phi'} illustrates how $\phi'(C)$ is obtained for a given hook configuration. Every point inserted to obtain C' has another hook endpoint one spot vertically below, so each underlined Y will always have either a Z or another Y right before it in the word, never an X.
Thus $\phi'(C) \in \underline{\Duck}_{k,i}$ for all $C \in \RedVHC_k(\Av_{3k-i}(312))$.

We claim that given $C'$ together with the list of points which were inserted we can reconstruct $C$. Indeed we simply remove all inserted points and re-normalize. For each hook whose southwest endpoint had been an inserted point, $p'$, its new southwest endpoint is now $p$, the point whose horizontal component is one less than $p'$. We will use this ability to recover $C$ to show that $\phi'$ is both injective and surjective.

\textbf{Injectivity of $\phi'$:} Suppose $\phi'(C_1)=\phi'(C_2)$. Ignoring which Y's are underlined they must both get sent to the same 3D-Dyck word, so $C_1'=C_2'=C'$. Given our chosen set of underlined Y's (and thus our known set of inserted points) we can recover the unique $C$ that yeilded $C'$, so $C_1=C_2$.

\textbf{Surjectivity of $\phi'$:} For any $\underline{w} \in \underline{\Duck}_{k,i}$ we can ignore which Y's are underlined for a moment to get a word $w \in \Dyck^3$. Taking $\phi^{-1}(w)$ gives us $C'$. Since no underlined Y has an X immediately before it, its corresponding SW endpoint must have another endpoint one spot vertically below it. We can remove this SW endpoints, making its hook instead start at the endpoint below. Doing this for all underlined Y's gives us a configuration $C$.  All descent bottoms with heights originally between two endpoints will still have heights between those two endpoints. We removed exactly $i$ southwest endpoints, so $C$ is a configuration on $3k-i$ points. All points in $C$ are either a hook endpoint, descent bottom or both. Removing points will never introduce a 312. Thus $C \in \RedVHC_k(\Av_{3k-i}(312))$ and $\phi'(C)= \underline{w}$.  
\end{proof}

\begin{figure}[h]\begin{center}
\begin{tikzpicture}[scale=.4]
	\coordinate (P1) at (1,4);
	\coordinate (P2) at (2,3);
	\coordinate (P3) at (3,2);
	\coordinate (P4) at (4,6);
	\coordinate (P5) at (5,7);
	\coordinate (P6) at (6,5);
	\coordinate (P7) at (7,9);
	\coordinate (P8) at (8,10);
	\coordinate (P9) at (9,8);
	\coordinate (P10) at (10,11);
	\foreach \pt in {P1, P2, P3, P4, P5, P6, P7, P8, P9, P10}
		\fill (\pt) circle (1.2mm);
	\foreach \a/\b in {P1/P8, P2/P4, P5/P7, P8/P10}
		\draw (\a) -- (\a)|-(\b) -- (\b);
		\draw[red] (P2) circle (5mm);
	\draw[red] (P8) circle (5mm);
		\draw (12,5) node[scale=2] {$\longrightarrow$};
	\coordinate (Q1) at (17,3);
	\coordinate (Q2) at (18,2);
	\coordinate (Q3) at (19,4);
	\coordinate (Q4) at (20,1);
	\coordinate (Q5) at (21,6);
	\coordinate (Q6) at (22,7);
	\coordinate (Q7) at (23,5);
	\coordinate (Q8) at (24,9);
	\coordinate (Q9) at (25,10);
	\coordinate (Q10) at (26,11);
	\coordinate (Q11) at (27,8);
	\coordinate (Q12) at (28,12);
	\foreach \pt in {Q1, Q2, Q3, Q4, Q5, Q6, Q7, Q8, Q9, Q10, Q11, Q12}
		\fill (\pt) circle (1.2mm);
	\foreach \a/\b in {Q1/Q9, Q3/Q5, Q6/Q8, Q10/Q12}
	    \draw (\a) -- (\a)|-(\b) -- (\b);
	\draw[red] (Q3) circle (5mm);
	\draw[red] (Q10) circle (5mm);
	\draw[blue] (Q2) circle (3mm);
	\draw[blue] (Q9) circle (3mm);
	\draw (23, 0) node[scale=2] {$\big\downarrow$};
	\draw (23,-3) node[scale=1] {XXYYXZYXZZYZ};
		\draw (14,-3) node[scale=2] {$\longleftarrow$};
	\draw (5,-2.5) node[scale=1] {XXYYXZYXZZYZ};
	\draw[red] (3.6,-3) node[scale=2] {-};
	\draw[red] (7.9,-3) node[scale=2] {-};
	\draw (28, -4) node {};
	\draw (28,12) node {};
\end{tikzpicture}
\caption{Illustration of $\phi'$ acting on a configuration $C \in \RedVHC_k(\Av_{3k-i}(312))$ as outlined in the proof of \cref{thrm:duck words bij}. First $C$, with points that are both SW endpoints and something else circled in red, is converted into a configuration $C' \in \RedVHC_k(\Av(312))$ now with added points circled in red and original points in smaller blue circles. Next $C'$ is sent to a 3D-Dyck word by $\phi(C')$, and finally the Y's corresponding to the added points are underlined to obtain $\phi'(C)$.} 
\label{fig:phi'}
\end{center}\end{figure}
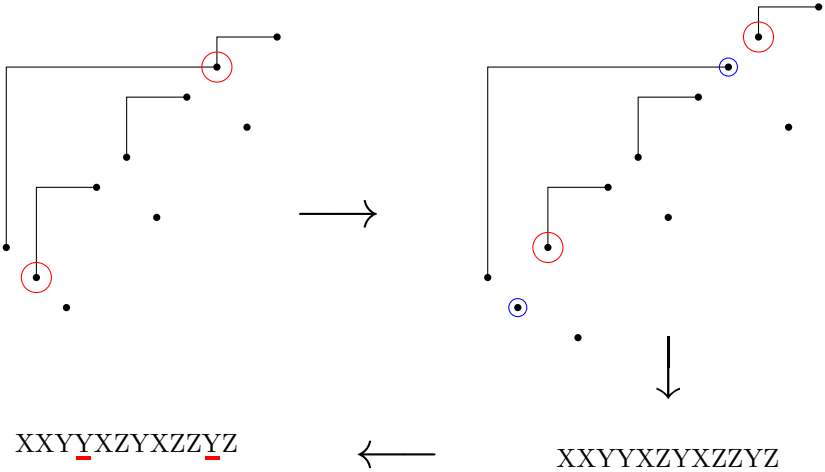

\section{Duck Words}
We now successfully have a way to encode any hook configuration in $\RedVHC_k(\Av(312))$ as an underlined duck word. In trying to count these words, it is useful to express them as a sum of non-underlined 3D-Dyck words in which varying numbers of Y's do or do not have an X immediately before them. This leads to a more intuitive class of combinatorial objects in terms of which we can express $\RedVHC_k(\Av(312))$. We use these words to prove several more properties about 312-avoiding reduced valid hook configurations, including further conjectures of Sankar's. We also note several patterns that arise from counting the words themselves.

\begin{definition}
The \emph{$(k,i)$-Duck words} are elements of $\Dyck^3_k$ in which exactly $i$ Y's do not have an X immediately before them. The number of these words is denoted by $\Duck_{k,i}$.
\end{definition}

The first few values of $\Duck_{k,i}$ are as follows, where $k$ is the row number, and $i$ increases by one from left to right, going from $0$ to $k-1$ 
\medskip

\noindent
$\begin{array}{ccccccc}
1 & \ & \ & \ & \  & \ & \ \\
2 &   3 & \ & \ & \  & \ & \ \\
5  &  23  &   14  & \  & \  & \ & \ \\
14  &  131 &     233   &    84  & \  & \ & \ \\
42  & 664   &  2339 &  2367   &   594  & \ & \ \\
132  & 3166   &   18520   &  36265   &   24714 &   4719 & \ \\
429 & 14545 & 127511 & 408311 & 527757 & 266219 &
 40898
\end{array}$

\begin{proposition}
The $k^{th}$ row sums in the above triangle are given by $$f_k(1) = \sum\limits_{i=0}^{k-1}\Duck_{k,i} = |\Dyck^3_k|.$$
\end{proposition}
\begin{proof}
Every 3D-Dyck word has $i$ Y's that do not have an X immediately before them for some value of $i$. Summing over all these values of $i$ clearly gives the total number of 3D-Dyck words.
\end{proof}

\begin{proposition}
We have $h_k(0) = \Duck_{k,0}=C_k$.
\end{proposition}
\begin{proof}
Every X comes exactly before each Y. Writing each word in $\Duck_{k,0}$ with a U to represent each appearance of XY and a D to represent each Z, we get a bijection with Dyck words.
\end{proof}

\begin{proposition}
\label{observation:Defant}
We have $$\Duck_{k,k-1}=|\RedVHC_k(\Av_{2k+1}(312))| = C_kC_{k+2}-C_{k+1}^2.$$
\end{proposition}
\begin{proof}
By \cref{thrm:duck words bij}, we know $|\RedVHC_k(\Av_{2k+1}(312))| = \underline{\Duck}_{k,k-1}$. \\
For any $C \in |\RedVHC_k(\Av_{2k+1}(312))|$, every SW endpoint after the first is also either a NE endpoint or descent bottom. Thus \cref{thrm:duck words bij} tells us that in $\phi'(C)$ the corresponding Y has either a Z or another Y before it, so we also have $|\RedVHC_k(\Av_{2k+1}(312))|=\Duck_{k,k-1}$. The second equality follows as a result of Corollary 5.1 of \cite{catalan} because, as Sankar noted in \cite{maya}, there is a natural bijection between the set $\RedVHC_k(\Av_{2k+1}(312))$ and the set of $312$-avoiding uniquely sorted permutations in $S_{2k+1}$.
\end{proof}

\begin{theorem}
\label{thrm:underline}
We have \[ |\RedVHC_k(\Av_{3k-i}(312))|=\underline{\Duck}_{k,i}=\sum\limits_{j=i}^{k-1}\Duck_{k,i}\cdot \binom{j}{i}. \]
\end{theorem}

\begin{proof}
Every word in $\underline{\Duck}_{k,i}$ must have at least $i$ Y's without an X before them. For each $j>i$ we can take a word in which exactly $j$ Y's do not have an X before them and choose $i$ of those Y's to underline. Summing over all the possible $j$'s gives all possible underlined duck words.
\end{proof}

Writing out each $\underline{\Duck}_{k,i}$ as a sum of the $\Duck_{k,j}$'s, we get the following set of equations:

\bigskip

\noindent
\begin{equation}
\begin{array}{cccccccc}
\underline{\Duck}_{k,0}   & = & \binom{0}{0}\Duck_{k,0}  \ \ \   + &  \binom{1}{0}\Duck_{k,1} \ \ \  + & \binom{2}{0}\Duck_{k,2} & \   + \ \cdots \ + & \binom{k-1}{0}\Duck_{k,k-1} \vspace{2mm}\\ 
%%%%%%%%%%%
\underline{\Duck}_{k,1}   & = &   & \binom{1}{1}\Duck_{k,1} \ \   + & \binom{2}{1}\Duck_{k,2} & \ + \ \cdots \ + & \binom{k-1}{1}\Duck_{k,k-1} \vspace{2mm}\\
%%%%%%%%%%%%%
\underline{\Duck}_{k,2}   & = &   &  & \binom{2}{2}\Duck_{k,2} & \  + \ \cdots \ + & \binom{k-1}{2}\Duck_{k,k-1} \vspace{2mm}\\
%%%%%%%%%%%%%%
\vdots   & = &   &  &  &    & \vdots \vspace{2mm}\\
%%%%%%%%%%%%
\underline{\Duck}_{k,k-1}   & = &   &  &  &    & \binom{k-1}{k-1}\Duck_{k,k-1} \vspace{2mm}.\\
\end{array}
\label{eqn:duck}
\end{equation} \\

\begin{corollary}
\label{corollary: sum of RedVHC_k}
We have $|\RedVHC_k(\Av(312))|= \sum\limits_{j=0}^{k-1}2^i \cdot \Duck_{k,j}$.
\end{corollary}

\begin{proof}
Observe that
\begin{align*}
|\RedVHC_k(\Av(312))| &=\sum_{i=0}^{k-1}\underline{\Duck}_{k,i} \\
&= \sum_{i=0}^{k-1} \sum_{j=i}^{k-1} \binom{j}{i}\Duck_{k,j} = \sum_{j=0}^{k-1} \sum_{i=0}^{j} \binom{j}{i}\Duck_{k,j} \\
&= \sum_{j=0}^{k-1}2^i \cdot \Duck_{k,j}. \qedhere
\end{align*}
\end{proof}

We now use a similar method to prove \cref{conj: Sankar 2}.

\begin{corollary} \label{corollary:Sankar 2}
We see that $f_k(-1) = \sum\limits_{i=0}^{k-1}(-1)^i |\RedVHC_k(\Av_{3k-i}(312))|= C_k$.
\end{corollary}

\begin{proof}
We have
\begin{align*}
\sum_{i=0}^{k-1}(-1)^i |\RedVHC_k(\Av_{3k-i}(312))| &=\sum_{i=0}^{k-1}(-1)^i \underline{\Duck}_{k,i} \\
&= \sum_{i=0}^{k-1} \sum_{j=i}^{k-1}(-1)^i \binom{j}{i}\Duck_{k,j}
= \sum_{j=0}^{k-1} \sum_{i=0}^{j} (-1)^i \binom{j}{i}\Duck_{k,j} \\
&= \binom{0}{0}\Duck_{k,0}+ \sum_{j=1}^{k-1}0 \cdot \Duck_{k,j} = C_k.
\ \ \ \qedhere
\end{align*}
\end{proof}

Also note that Corollaries \ref{corollary: sum of RedVHC_k} and \ref{corollary:Sankar 2} can be clearly seen to be true by taking the sum or alternating sum respectively of the elements in each column of \cref{eqn:duck}.

In addition to finding specific values, a combinatorial interpretation for the coefficients of $h_k(x)$ follows as a result of \cref{thrm:underline}.

\begin{corollary}
\label{corollary:polynomail}
$h_k(x)= \sum\limits_{i=0}^{k-1}\Duck_{k,i} \cdot x^i$.
\end{corollary}
\begin{proof}
We need only show that $\sum\limits_{i=0}^{k-1}\Duck_{k,i} \cdot (x+1)^i=f_k(x)$. This follows from

\begin{align*}
\sum_{i=0}^{k-1} \Duck_{k,i} \cdot (x+1)^i &= \sum\limits_{i=0}^{k-1}\sum\limits_{j=1}^{i}\Duck_{k,i}\cdot \binom{i}{j} \cdot x^j \\
&=\sum\limits_{j=0}^{k-1}\sum\limits_{i=j}^{k-1}\Duck_{k,i}\cdot \binom{i}{j} \cdot x^j = \sum\limits_{j=0}^{k-1}\underline{\Duck}_{k,j}\cdot x^j\\
&= f_k(x) = h_k(x+1).
\qedhere
\end{align*}
\end{proof}

\cref{corollary:polynomail} actually proves a stronger version of \cref{conj: Sankar 3}, which was that $h_k(x)$ has strictly positive coefficients. \cref{corollary:polynomail} gives a nice combinatorial interpretation for the coefficients of $h_k(x)$ that is simpler than that of the coefficients of $f_k(x)$.

\subsection{Rewritten Duck Words}
\label{section:rewritten}

Without losing any information we can underline each Y without an X immediately before it, and delete the X immediately before each non-underlined Y. Now we just rewrite each Y as a U and each Z as a D, leaving the underlines where they were. We replace each X with a circle around the first non-X letter that  comes after it. That letter will never be an underlined U. This results in a rewriting of our duck word which we can ``decode" to get the original form back.

For example, following this rewriting process, 

 XYXYZXXYZYZ \ $\longrightarrow$ \ YYZXYZ\underline{Y}Z \ $\longrightarrow$ UUDXUD\underline{U}D $\longrightarrow$ UUD\textcircled{U}D\underline{U}D.

\begin{definition}
We call the words rewritten according to the rule above \textbf{rewritten $\bf{(k,i)}$-duck words}.
\end{definition}

The name ``duck words" comes from the rewriting which replaces Y's with U's.

Using this method, the set of rewritten $(k,i)$-duck words is exactly the set of words in $\Dyck_k$ with $i$ underlined Y's and $i$ circles around non-underlined letters such that up to any point in the word the number of circles is greater than or equal to the number of underlines. Letters can have more than one circle around them.

In general, the rewritten duck words may not seem like much of a simplification over duck words in the normal form, but in the case of small $i$'s it can be more simple to consider the rewritten form. We will do this specifically for $\Duck_{k,1}$ in the following section.

\subsection{Weighted Tennis Ball Numbers}

We next show that a weighted version of the tennis ball numbers described by Merlini, Sprugnoli and Verri in \cite{TB} are equal to the numbers $\Duck_{k,1}$. They provide the following description of the tennis ball numbers:

\noindent Put tennis balls 1 and 2 in a room. Throw one of them at random onto the lawn. \\
Put tennis balls 3 and 4 into the room. Throw one of the three tennis balls now in the room onto the lawn. \\
Put tennis balls 5 and 6 in the room. Throw one of the tennis balls from the room onto the lawn. \\
Repeat until you have thrown tennis balls $2k-1$ and $2k$ into the room and then thrown a tennis ball from the room onto the lawn, resulting in a total of $k$ balls on the lawn.

\begin{definition}
Let $TB_k$ be the set of possible ball configurations on the lawn as described by \cite{TB}. The $k^{th}$ tennis ball number is $|TB_k|$, and the $k^{th}$ weighted tennis ball number, denoted $tb_k$, is the sum of the value of the balls on the lawn over all possible configurations in $TB_k$. It is known that $tb_k$ is given by the formula $	
\frac{(2n^2 + 5n + 4)\binom{2n+1}{n}}{(n+2)} - 2^{2n+1}$ \cite{tn}.
\end{definition}

Grimaldi and Moser \cite{tennisball} proved that the set $TB_{k-1}$ is in bijection with Dyck words of length $2k$. We can think of this bijection as a map $\psi: TB_{k-1} \rightarrow \Dyck_k$ which makes the $(i+1)^{st}$ letter a U exactly when the $i^{th}$ tennis ball is on the lawn. For example, if there are six balls total and the ball configuration on the lawn is $a=\{2,3,5\}$, then $\psi(a) =$ UDUUDUDD. We will use this bijection to show that $\Duck_{k,1}$ is given by $tb_{k-1}$.

\begin{proposition}
\label{proposion:tb_k}
We have $\Duck_{k,1} = tb_{k-1}$.
\end{proposition}
\begin{proof}
For each ball configuration $a \in TB_{k-1}$ we can consider the sum of the ball numbers, which is the amount it contributes to $tb_{k-1}$. Call this number $C_{TB}(a)$. Similarly for the Dyck path $\psi(a) \in \Dyck_k$ we can consider the total number of ways to underline one of the U's and circle a letter before it, which gives its contribution to $\Duck_{k,1}$, which we will denote by $C_{Dyck}(\psi(a))$. We will show that for $a \in TB_{k-1}$ and $\psi(a) \in \Dyck_k$, we have $C_{TB}(a)=C_{Dyck}(\psi(a))$.

Let $B_{\psi(a)}(U_i)$ be the number of letters before the $i^{th}$ U in the word $\psi(a)$. For $i \in [k]$, the number of ways to underline the $(i+1)^{st}$ U in $\psi'(a)$ and circle a letter before it is exactly $B_{\psi(a)}(U_{i+1})$. Then $$C_{Dyck}(\psi(a))=\sum\limits_{i=1}^{k-1} B_{\psi(a)}(U_{i+1}).$$
For $j_i$ the number on the $i^{th}$ ball in $a$, we know $\psi(a)$ will put the $(i+1)^{st}$ U in the $(j_i+1)^{st}$ spot in the word. Thus we have that $j_i=B_{\psi(a)}(U_{i+1})$. Summing over all balls in $a$ gives us $$C_{TB}(a)=\sum\limits_{i=1}^{k-1}j_i=\sum\limits_{i=1}^{k-1} B_{\psi(a)}(U_{i+1}).$$

Summing over the contributions from every ball arrangement and word, we conclude that

$$tb_{k-1}=\sum\limits_{a \in TB_{k-1}}C_{TB}(a)=\sum\limits_{\psi(a) \in \Dyck_k}C_{Dyck}(\psi(a))=\Duck_{k,1}. \qedhere$$
\end{proof}

The linear term of $h_k(x)$ is $\Duck_{k,1}$, so \cref{proposion:tb_k} proves \cref{conj: Sankar 4}.

\section{Future Work}
In this paper we expressed 312-avoiding reduced valid hook configurations in terms of underlined and regular duck words. We also showed that these give the coefficients of $f_k(x)$ and $h_k(x)=f_k(x-1)$ respectively. We evaluated $f_k(x)$ for $x=-1,0,1$ and $h_k(x)$ for $x=0,1,2$, as well as found the constant and leading order terms $f_k(x)$ and $g_k(x)$, and the linear term of $h_k(x)$. There are still many directions open to pursue. We pose several conjectures and open problems.

\begin{Problem}
Find the values of $f_k(x)$ and $g_k(x)$, or combinatorial interpretations of these numbers, for further values of $x$.
\end{Problem}

\begin{conjecture}[Sankar \cite{maya}]
For each positive integer $k$, the polynomial $h_k(x)$ (and therefore $f_k(x)$) has all real roots.
\end{conjecture}

For a weaker but still interesting result, one could show that the coefficients of one or both of these polynomials are unimodal or log-concave.

\begin{definition}
Let $\mathcal{W}_n$ be the set of closed walks consisting of $2n$ unit steps north, east, south, or west starting and ending at the origin and confined to the first octant.
\end{definition}

\begin{Problem}
Find a direct bijection between the set of $(k,k-1)$-Duck words and $\mathcal{W}_n$.
\end{Problem}

This would provide a simpler and more direct proof of \cref{observation:Defant} than the one given by Corollary 5.1 of \cite{catalan}.

\begin{question}
What other combinatorial objects are in bijection with 312-avoiding reduced valid hook configurations? What other combinatorial objects are in bijection with Duck words?
\end{question}

\begin{Problem}
Find a formula for $\Duck_{k,i}$.
\end{Problem}

This would in turn yeild a formula for 312-avoiding reduced valid hook configurations as well as the total number of 312-avoiding valid hook configurations, whose generating function Sankar has proved to not be D-finite \cite{maya}.

\section{Acknowledgments}

This research was conducted at the REU at the University of Minnesota Duluth, which is supported by
NSF/DMS grant 1949884 and NSA grant H98230-20-1-0009. I want to thank Joe Gallian for running the REU, Yelena Mandelshtam and Colin Defant for proofreading, and Colin Defant, Amanda Burcroff, and Yelena Mandelshtam for advising the program and giving guidance on my problem. I would also like to thank Maya Sankar for posing the problem and helping with code to count valid hook configurations.

\end{document}